\documentclass[reqno, 11pt]{amsart}
\usepackage{verbatim}

\newcommand{\mysection}[1]{\section{#1}
\setcounter{equation}{0}}

\newtheorem{theorem}{Theorem}[section]
\newtheorem{corollary}[theorem]{Corollary}
\newtheorem{lemma}[theorem]{Lemma}
\newtheorem{proposition}[theorem]{Proposition}

\theoremstyle{definition}

\theoremstyle{definition}

\theoremstyle{definition}

\makeatletter
\def\dashint{\operatorname%
{\,\,\text{\bf--}\kern-.98em\DOTSI\intop\ilimits@\!\!}}
\makeatother

\renewcommand{\epsilon}{\varepsilon}

\def\bR{\mathbb{R}}

\def\bH{\mathbb{H}}

\def\ff{\mathfrak{f}}
\def\fg{\mathfrak{g}}

\def\cC{\mathcal{C}}
\def\cD{\mathcal{D}}

\def\cH{\mathcal{H}}
\def\cP{\mathcal{P}}

\def\cV{\mathcal{V}}

\def\cL{\mathcal{L}}

\def\cS{\mathcal{S}}

\newcommand{\set}[1]{\left\{#1\right\}}

\newcommand{\Div}{\operatorname{div}}

\begin{document}
\title[parabolic and elliptic systems]{Boundary Gradient estimates for parabolic and elliptic systems from linear laminates}

\author[H. Dong]{Hongjie Dong}
\address[H. Dong]{Division of Applied Mathematics, Brown University,
182 George Street, Providence, RI 02912, USA}
\email{Hongjie\_Dong@brown.edu}
\thanks{H. Dong was partially supported by NSF grant number DMS-1056737.}

\author[J. Xiong]{Jingang Xiong}
\address[J. Xiong]{Beijing International Center for Mathematical Research, Peking University, Beijing 100871, China}
\email{jxiong@math.pku.edu.cn}
\thanks{J. Xiong was partially supported by the First Class Postdoctoral
Science Foundation of China (No. 2012M520002).}



\begin{abstract}
We study boundary gradient estimates for second-order divergence type parabolic and elliptic systems in $C^{1,\alpha}$ domains. 
The coefficients and data are assumed to be H\"older in the time variable and all but one spatial variables. This type of systems
arises from the problems of linearly elastic laminates and composite materials.
\end{abstract}

\maketitle

\mysection{Introduction}
                                        \label{secIntro}

In this paper, we consider parabolic and elliptic systems:
\begin{align}
                            \label{parabolic0}
\cP u&:=-u_t+D_\alpha(A^{\alpha\beta}D_\beta u)+D_\alpha(B^{\alpha}u)+\hat B^{\alpha}D_\alpha u+ C u=\Div g + f,\\
\cL u&:=D_\alpha(A^{\alpha\beta}D_\beta u)+D_\alpha(B^{\alpha}u)+\hat B^{\alpha}D_\alpha u+ C u=\Div g + f\label{elliptic0}
\end{align}
in a $C^{1,\delta}$ domain with the zero Dirichlet boundary condition.
The coefficients of $\cP$ and $\cL$ are assumed to be bounded, and the
operators are uniformly nondegenerate.
The \textbf{aim} of our paper is to study boundary Lipschitz regularity of weak solutions when coefficients
are assumed to be regular in the time variable and all but one spatial variables. This type of system
arises from the problems of linearly elastic laminates and composite materials.
The interior regularity has been systematically studied by Chipot, Kinderlehrer, and Vergara-Caffarelli \cite{CKC}, and also recently by the first author \cite{Dong12}. Particularly, it is known that weak solutions of \eqref{parabolic0} or \eqref{elliptic0} are Lipschitz in the interior of the domain.
Application backgrounds and some interior regularity results can also be found in
 \cite{BASL, BV, LiVo, LiNi, FKNN13, LiLi, XB} and many others.

Due to the interaction from the boundary data, solutions of these systems become more irregular near the boundary so that we cannot expect to
control the $C^{\gamma}$ norm of the solutions in the spatial variables for  $\gamma$ close to $1$, even when the given boundary data are smooth. See section 5.3 of Avellaneda-Lin \cite{AL87} for more discussions.

We shall establish gradient estimates in some boundary cone region and global H\"older estimates for system \eqref{parabolic0}.
To be more precise, let us denote a typical point in ${\bR}^{d+1}$ by $z=(t,x)$, where $x=(x^1,\ldots,x^d):=(x',x^d)$ and $z':=(t,x')$. Let $\Omega$ be a $C^{1,\delta}$ domain in $\bR^d$, $0\in \partial\Omega$, and $e_d:=(0,\ldots,0,1)$  be the inner normal direction of $\partial\Omega$ at $0$.
We prove that if the coefficients and data are H\"older continuous in $z'$,
then any local weak solution $u$ to \eqref{parabolic0} near the origin with the zero boundary
condition is Lipschitz continuous with respect to all spatial variables for $x\in B_\varepsilon\cap \{x^d\ge |x'|\}$
with some small $\varepsilon>0$. Moreover, $u$ is $C^{1/2}$ in $t$ and $D_{x'}u$ and $ U:=A^{d\beta}D_\beta u+B^d u-g_d$ are
H\"older continuous in the same region; see Theorem \ref{thm1} below for a more precise statement. A similar result holds true for the elliptic system \eqref{elliptic0}.

For the proof of Theorem \ref{thm1}, we follow an argument in \cite{Dong12} which is based on Campanato's approach. The classical Campanato's
approach is to show that the mean oscillations of $Du$ in balls vanish in a certain order
as the radii of balls go to zero. Due to the lack of regularity of the coefficients and data in the $x^d$-direction, it is not possible to control the mean oscillations of the whole gradient of $u$. In fact, $D_d u$ may be discontinuous with respect to $x^d$. As in \cite{Dong12}, instead we estimate the mean oscillations of $D_{x'}u$ and $U$, and use a certain decomposition of $u$. Here an added difficulty comes from the curved boundary. The usual argument of flattening the boundary does not seem to be plausible because of the irregularity of the coefficients. To this end, we exploit an idea recently used in \cite{DK11} to locally ``cut off'' the boundary, so that the modified equation is satisfied in a domain with a piece of flat boundary and the difference between the solutions can be estimated by using the Hardy inequality. Another ingredient of our proof is a boundary $L_p$ estimate for systems with partially VMO coefficients proved in the same paper \cite{DK11}.

Notice that in general Theorem \ref{thm1} may not hold if the inner normal drection of $\partial \Omega$ at $0$ is not paralleled to $e_d$. In this case, we shall prove a H\"older regularity; see Theorem \ref{thm2}. Of independent interest we obtain a H\"older regularity for parabolic systems with coefficients of VMO in the time variable and $d-2$ spatial variables; see Proposition \ref{prop5.1}. In the proof, we again use Campanato's approach as well as an anisotropic Sobolev inequality and the reverse H\"older inequality.

The organization of this paper is as follows. In Section \ref{sec2}, we state our main theorems for divergence form systems and introduce some notation. We prove some auxiliary estimates in Section \ref{sec3}. The proofs of main theorems are given in Sections \ref{sec4} and \ref{sec5}.

\mysection{Notation and main results}
                                            \label{sec2}

We are concerned with parabolic systems
\begin{equation}
                                                \label{parabolic}
\cP u:=-u_t+D_\alpha(A^{\alpha\beta}D_\beta u)+D_\alpha(B^{\alpha}u)+\hat B^{\alpha}D_\alpha u+ C u=\Div g + f
\end{equation}
in a cylindrical domain $(-\infty,0)\times \Omega$ with the zero Dirichlet boundary condition on the lateral boundary.
The coefficients $A^{\alpha\beta}$, $B^\alpha$, $\hat B^\alpha$, and $C$ are $n \times n$ matrices,
which are bounded by a positive constant $K$, and
the leading coefficient matrices $A^{\alpha\beta}$ are strongly elliptic with ellipticity constant $\nu$:
$$
\nu|\xi|^2\le A^{\alpha\beta}_{ij}\xi^\alpha_i\xi^\beta_j,\quad |A^{\alpha\beta}|\le \nu^{-1}
$$
for any $\xi=(\xi^\alpha_i)\in \bR^{d\times n}$.
Here $g=(g_1,\ldots,g_d)$, and
\begin{align*}
u& = (u^1, \ldots, u^n)^{\text{tr}},\\
g_{\alpha}& = (g^1_\alpha, \ldots, g^n_\alpha)^{\text{tr}}\quad\text{for}\,\,\alpha=1,\ldots,d,\\
f& = (f^1,\ldots,f^n)^{\text{tr}}
\end{align*}
are (column) vector-valued functions. Throughout the paper, the summation convention
over repeated indices is used. We also consider the following elliptic system
\begin{equation}
                                                \label{elliptic}
\cL u:=D_\alpha(A^{\alpha\beta}D_\beta u)+D_\alpha(B^{\alpha}u)+\hat B^{\alpha}D_\alpha u+ C u=\Div g + f
\end{equation}
in $\Omega$ with the zero Dirichlet boundary condition.
In this case $A^{\alpha\beta}$, $B^\alpha$, $\hat B^\alpha$, $C$, $g$, and $f$
are independent of $t$ and satisfy the same conditions as in the parabolic case.
\subsection{Notation}
By $Du=(D_i u)$ and $D^{2}u=(D_{ij}u)$ we mean the gradient and the Hessian matrix
of $u$. On many occasions we need to take these objects relative to only
part of variables. We also use the following notation:
$$
D_tu=u_t,\quad D_{x'}u=u_{x'},\quad
D_{xx'}u=u_{xx'}.
$$

Set
$$
B_r'(x') = \{ y \in \bR^{d-1}: |x'-y'| < r\}, \quad
B_r(x) = \{ y \in \bR^d: |x-y| < r\},
$$
$$
Q_r'(t,x) = (t-r^2,t) \times B_r'(x'),\quad
Q_r(t,x) = (t-r^2,t) \times B_r(x),
$$
and
$$
B_r'=B_r'(0),\quad
B_r = B_r(0),\quad
Q_r'=Q_r'(0,0),\quad
Q_r=Q_r(0,0),
$$
$$
B_r^+=B_r\cap \{x^d>0\},\quad Q_r^+=Q_r\cap \{x^d>0\}.
$$
For a domain $\Omega\in \bR^d$ and $r>0$, we denote
\begin{align*}
\Omega_r(x)&=\Omega\cap B_r(x),\quad
\cD_r(z)=(t-r^2,t)\times \Omega_r(x),\\
\Gamma_r(z)&=(t-r^2,t)\times(B_r(x)\cap \partial \Omega),\\
\cC_r(t,x)&=B_r(x)\cap \{y\in \bR^d: y^d> |y'|\},\\
\cV_r(t,x)&=(t-r^2,t)\times \cC_r(t,x).
\end{align*}
As before, we use the abbreviations $\Omega_r$, $\cD_r$, $\Gamma_r$, etc.

By $N(d,p,\ldots)$ we mean that $N$ is a constant depending only
on the prescribed quantities $d, p,\ldots$.
For a (matrix-valued) function $f(t,x)$ in $\bR^{d+1}$, we set
\begin{equation*}
(f)_{\cD} = \frac{1}{|\cD|} \int_{\cD} f(t,x) \, dx \, dt
= \dashint_{\cD} f(t,x) \, dx \, dt,
\end{equation*}
where $\cD$ is an open subset in $\bR^{d+1}$ and $|\cD|$ is the
$d+1$-dimensional Lebesgue measure of $\cD$.

\subsection{Lebesgue spaces}

For $p\in (1,\infty)$, we denote $\bH^{-1}_{p}(\cD)$ to be the space consisting of all functions $u$ satisfying
$$
\inf\set{\|g\|_{L_{p}(\cD)}+\|h\|_{L_{p}(\cD)}\,|\,u=\Div g+h}<\infty.
$$
It is easily seen that $\bH^{-1}_{p}(\cD)$ is a Banach space. Naturally, for any $u\in \bH^{-1}_{p}(\cD)$, we define the norm
\begin{equation*}
              %                              \label{eq10.36pm}
\|u\|_{\bH^{-1}_{p}(\cD)}=\inf\set{\|g\|_{L_{p}(\cD)}+\|h\|_{L_{p}(\cD)}\,|\,u=\Div g+h}.
\end{equation*}
We also define
$$
\cH^{1}_{p}(\cD)=
\set{u:\,u,Du \in L_{p}(\cD),u_t\in \bH^{-1}_{p}(\cD)}.
$$
We use the abbreviations $\cH^{1}_{p}=\cH^{1}_{p}(\bR^{d+1})$, etc.

\subsection{Partially VMO and H\"older spaces}  \label{ssec2.3}
For a function $u$ in $\bR^{d+1}$, we define  its modulus of continuity $\omega_{u,z'}$ (in the mean) with respect to $z'$ by
\begin{align*}
&\omega_{u,z'}(R)\\&\,=\sup_{r\le R}\sup_{z_0\in \bR^{d+1}}\left(
%|Q_r|^{-2}
\dashint_{Q_r(z_0)}\dashint_{Q_r(z_0)}|u(t,x',x^d)-u(s,y',x^d)|^2\,dy\,ds\,dx\,dt
\right)^{\frac 1 2}.
\end{align*}
We say $u$ is partially VMO with respect to $z'$ if $\omega_{u,z'}(R)\to 0$ as $R\to 0$.

For $\delta\in (0,1]$ and a function $u$ in $\cD\subset\bR^{d+1}$, we denote its $C^{\delta/2,\delta}$ semi-norm by
$$
[u]_{\delta/2,\delta;\cD}:=\sup_{\substack{(t,x),(s,y)\in \cD\\ (t,x)\neq (s,y)}}\frac {|u(t,x)-u(s,y)|}{|t-s|^{\delta/2}+|x-y|^\delta},
$$
and its $C^{\delta/2,\delta}$ norm by
$$
|u|_{\delta/2,\delta;\cD}:=[u]_{\delta/2,\delta;\cD}+
| u|_{0;\cD},
$$
where $|u|_{0;\cD}=\sup_{\cD}|u|$.

We define a partial H\"older semi-norm with respect to $z'$ as
\[
[u]_{z',\delta/2,\delta;\cD}:=\sup_{\substack{(t,x),(s,y)\in \cD\\x^d=y^d,(t,x)\neq (s,y)}}\frac {|u(t,x)-u(s,y)|}{|t-s|^{\delta/2}+|x-y|^\delta},
\]
and the corresponding norm as
\[
|u|_{z',\delta/2,\delta;\cD}:=[u]_{z',\delta/2,\delta;\cD}+|u|_{0;\cD}.
\]
By $C^{\delta/2,\delta}_{z'}(\cD)$ we denote the set of all bounded measurable functions $u$ on $\cD$ for which $[u]_{z', \delta/2,\delta;\cD}<\infty$. In the time-independent case, we define $[\cdot]_{x',\delta}$, $|\cdot|_{x',\delta}$, and the space $C^\delta_{x'}$ in a similar fashion.

To get a global H\"older estimate, we also deal with coefficients which are measurable in two directions. Denote $x''=(x^1,\ldots,x^{d-2})$ and $z''=(t,x'')$. For a function $u$ in $\bR^{d+1}$, we define its modulus of continuity $\omega_{u,z''}$ (in the mean) with respect to $z''$ by
\begin{multline*}
\omega_{u,z''}(R)
%|Q_r|^{-2}
=\sup_{r\le R}\sup_{z_0\in \bR^{d+1}}\Big(\dashint_{Q_r(z_0)}\dashint_{Q_r(z_0)}|u(t,x'',x^{d-1},x^d)\\
-u(s,y'',x^{d-1},x^d)|^2\,dy\,ds\,dx\,dt
\Big)^{\frac 1 2}.
\end{multline*}
We say that $u$ is partially VMO with respect to $z''$ if $\omega_{u,z''}(R)\to 0$ as $R\to 0$.
We remark that by the triangle inequality and the H\"older inequality, for any $z_0\in \bR^{d+1}$ and $R>0$ we can find $\bar u=\bar u(x^{d-1},x^d)$ such that
\begin{equation*}
                    %\label{eq2.03}
\dashint_{Q_R(z_0)}|u(t,x'',x^{d-1},x^d)-
\bar u(x^{d-1},x^d)|\,dx\,dt\le N(d)\,\omega_{u,z''}(\sqrt 2 R).
\end{equation*}

\subsection{Main results}

We state the main results of the paper concerning divergence form parabolic systems. In the sequel, we assume that $\Omega$ is a $C^{1,\delta}$ domain in $\bR^d$ for some $\delta\in (0,1)$, $0\in \partial \Omega$, and $e_d=(0,\ldots,0,1)$ is the inner normal direction of $\partial\Omega$ at $0$. The theorem reads that if the coefficients and data are partially H\"older continuous in $z'$, then any weak solution $u$ to \eqref{parabolic} is Lipschitz in all spatial variables and $1/2$-H\"older in $t$ in a cone with vertex at the origin.

\begin{theorem}
                            \label{thm1}
Let $\delta\in (0,1)$, $A\in C_{z'}^{\delta/2,\delta}$, $B\in C_{z'}^{\delta/2,\delta}$, $f\in L_\infty(\cD_1)$, and $g\in C_{z'}^{\delta/2,\delta}(\cD_1)$. Let $u$ be a weak solution to \eqref{parabolic} in $\cD_1$ and $u=0$ on $\Gamma_1$. Then there exists a constant $\varepsilon=\varepsilon(\Omega)>0$ such that $\cV_\varepsilon\subset \cD_1$, $u\in C^{1/2,1}(\cV_\varepsilon)$, $D_{x'} u, U\in C^{\delta/2, \delta}(\cV_\varepsilon)$ and we have
\begin{equation*}
                %                    \label{eq4.20}
|u|_{1/2,1;\cV_\varepsilon} +|D_{x'}u|_{\delta/2,\delta; \cV_\varepsilon} +|U|_{\delta/2,\delta; \cV_\varepsilon}
\le N(|g|_{z',\delta/2,\delta;\cD_{1}}+|f|_{0;\cD_{1}}
+\|u\|_{L_2(\cD_1)}),
\end{equation*}
where $N=N(d,n,\delta,\nu,K,\Omega,[A]_{z',\delta/2,\delta},[B]_{z',\delta/2,\delta})$, and
\[
U=A^{d\beta }D_\beta u+ B^d u-g_d.
\]
\end{theorem}

It follows from \cite{DK11} (see also Corollary \ref{cor3.1}  below) that for any $\gamma\in (0,1)$ there exist a
constant $\varepsilon_1=\varepsilon_1(d,n,\nu,\Omega,\gamma)$ such that $u\in C^{\gamma/2,\gamma}(\cD_{\varepsilon_1})$ and
\begin{equation}
             \label{eq:flatholder}
|u|_{\gamma/2,\gamma;\cD_{\varepsilon_1}}
\le N(\|g\|_{L_\infty(\cD_1)}+\|f\|_{L_\infty(\cD_1)}+\|u\|_{L_2(\cD_1)}),
\end{equation}
where $N=N(d,n,\nu,\Omega,K,\omega_{A,z'},\gamma)$. For this estimate, instead of the H\"older continuity we only require the leading coefficients $A^{\alpha\beta}$ to be partially VMO with respect to $z'$.

Next we fix a boundary point $x_0\in \partial\Omega$ such that the inner normal direction of $\partial\Omega$ at $x_0$ is not parallelled to $e_d$. Without loss of generality, we may assume that the normal direction lies in the plane spanned by $e_{d-1}$ and $e_d$. In this case, although in general $u$ is not Lipschitz near $z_0:=(0,x_0)$, we prove that it is H\"older continuous in a neighborhood of $z_0$. For scalar equations, this is a simple consequence of the De Giorgi--Nash--Moser estimate, for which no regularity assumptions on $A$ are needed.

\begin{theorem}
                            \label{thm2}
Suppose that $A$ is partially VMO with respect to $z''$, $g\in L_{p_0}(\cD_1(z_0))$ for some $p_0>d+2$, and $f\in L_{p_0^*}(\cD_1(z_0))$, where
$$
p_0^*=p_0(d+2)/(p_0+d+2).
$$
Let $u$ be a weak solution to \eqref{parabolic} in $\cD_1(z_0)$ and $u=0$ on $\Gamma_1(z_0)$. Then there exist constants $\delta_1=\delta_1(d,n,\nu,p_0)\in (0,1)$ and $\varepsilon=\varepsilon(\Omega)\in (0,1)$ such that for any $\delta\in (0,\delta_1)$, we have $u\in C^{\delta/2,\delta}(\cD_\varepsilon(z_0))$. Moreover, we have
\begin{equation}
                                    \label{eq9.24}
|u|_{\delta/2,\delta;\cD_\varepsilon(z_0)}
\le N(\|g\|_{L_{p_0}(\cD_1(z_0))}+\|f\|_{L_{p_0^*}(\cD_1(z_0))}
+\|u\|_{L_2(\cD_1(z_0))}),
\end{equation}
where $N=N(d,n,\delta,\nu,K,p_0,\Omega,x_0,\omega_{A,z''})$.
\end{theorem}

Combining \eqref{eq:flatholder}, \eqref{eq9.24}, and the corresponding interior estimates in \cite{Dong12}, we conclude that $u$ is globally H\"older continuous provided that $A\in C_{z'}^{\delta/2,\delta}$, $B\in C_{z'}^{\delta/2,\delta}$, $f\in L_\infty$, and $g\in C_{z'}^{\delta/2,\delta}$.

By modifying the proofs of Theorems  \ref{thm1} and \ref{thm2}, we also obtain the corresponding results for the elliptic system \eqref{elliptic}. In fact, Corollary \ref{cor1e} below is a simple consequence of Theorem \ref{thm1} by viewing the solutions to the elliptic systems as steady state solutions to the corresponding parabolic systems. We leave the details to the interested reader.

\begin{corollary}
                            \label{cor1e}
Let $\delta\in (0,1)$, $A\in C_{x'}^{\delta}$, $B\in C_{x'}^{\delta}$, $f\in L_\infty(\Omega_1)$, and $g\in C_{x'}^{\delta}(\Omega_1)$.  Let $u$ be a weak solution to \eqref{elliptic} in $\Omega_1$ and $u=0$ on $B_1\cap \partial \Omega$. Then there exists a constant $\varepsilon=\varepsilon(\Omega)>0$ such that $\cC_\varepsilon\subset \Omega_1$, $u\in C^1(\cC_\varepsilon)$, and we have
\begin{equation*}
|u|_{1;\cC_\varepsilon}
+|D_{x'}u|_{\delta;\cC_\varepsilon}+|U|_{\delta;\cC_\varepsilon}
\le N(|g|_{x',\delta;\Omega_1}+|f|_{0;\Omega_1}
+\|u\|_{L_2(\Omega_1)}),
\end{equation*}
where $N=N(d,n,\delta,\nu,K,\Omega,[A]_{x',\delta},[B]_{x',\delta})$ and $U$ is defined in Theorem \ref{thm1}.
\end{corollary}

\begin{corollary}
                            \label{cor2e}
Suppose that $A$ is partially VMO with respect to $x''$, $g\in L_{p_0}(\Omega_1(x_0))$ for some $p_0>d$, and $f\in L_{p_0^*}(\Omega_1(x_0))$, where
$$
p_0^*=p_0 d/(p_0+d).
$$
Let $u$ be a weak solution to \eqref{elliptic} in $\Omega_1(x_0)$  and $u=0$ on $B_1(x_0)\cap \partial \Omega$. Then there exist constants $\delta_1=\delta_1(d,n,\nu)\in (0,1)$ and $\varepsilon=\varepsilon(\Omega)\in (0,1)$ such that for any $\delta\in (0,\delta_1)$, we have $u\in C^{\delta}(\Omega_\varepsilon(x_0))$. Moreover, we have
\begin{equation*}
              %                      \label{eq9.24b}
|u|_{\delta;\Omega_\varepsilon(x_0)}
\le N(\|g\|_{L_{p_0}(\Omega_1(x_0))}+\|f\|_{L_{p_0^*}(\Omega_1(x_0))}
+\|u\|_{L_2(\Omega_1(x_0))}),
\end{equation*}
where $N=N(d,n,\delta,\nu,K,p_0,\Omega,x_0,\omega_{A,x''})$.
\end{corollary}

\mysection{Some auxiliary estimates}
                            \label{sec3}

We will use the following variant of the parabolic Poincar\'e inequality.

\begin{lemma}
                                        \label{paraPoin}
Let $p \in (1,\infty)$, $\Omega$ a Lipschitz domain, and $r\in (0,\text{diam}(\Omega))$. Assume that $u \in \cH^1_{p}(\cD_r(z_0))$ for some $z_0\in (-\infty,0)\times \Omega$.
Suppose that $B=\hat B=C=0$, $g,f\in L_{p,\text{loc}}$, and $\cP u=\Div g+f$ in $\cD_r(z_0)$. Then
\begin{equation*}
              %                              \label{eq22.58}
\int_{\cD_r(z_0)}|u(t,x)-(u)_{\cD_r(z_0)}|^p\,dz\le Nr^{p}\int_{\cD_r(z_0)}(|Du|^p+|g|^p+r^p |f|^p)\,dz,
\end{equation*}
where $N=N(d,\nu,p,\Omega)>0$.
\end{lemma}
\begin{proof}
In the special case when $\Omega=\bR^d$, this is Lemma 3.1 of \cite{Krylov_2005}. The general case follows from the same argument with obvious modifications, by noting that in each $\Omega_r(x_0)$ one can find a cutoff function $\zeta\in C_0^\infty(\Omega_r(x_0))$ such that $r|D\zeta|$ is uniformly bounded by a constant depending only on $d$ and the Lipschitz norm of $\partial\Omega$.
\end{proof}

Next we shall present some local $L_p$-estimates, which are deduced from the results obtained in \cite{DK11}.
\begin{lemma}
                                            \label{lem3.1}
Let $p\in (1,\infty)$. Assume that $A^{\alpha\beta}$ are partially VMO in $z'$ and $u\in \cH^1_p(\cD_1)$ satisfies
\begin{equation}
                                        \label{eq3.12}
\cP u=\Div g+f,
\end{equation}
in $\cD_1$ and $u=0$ on $\Gamma_1$, where $f,g\in L_p(\cD_1)$. Then there exist constants $r_0=r_0(d,n,\nu,\Omega,p)\in (0,1)$ and $N=N(d,n,\nu,\Omega,K,\omega_{A,z'},p)$ such that
\begin{equation*}
              %                      \label{eq2.43}
\|u\|_{\cH_p^1(\cD_{r_0})}\le N(\|u\|_{L_p(\cD_1)}+\|g\|_{L_p(\cD_1)}+\|f\|_{L_p(\cD_1)}).
\end{equation*}
\end{lemma}
\begin{proof}
Since $\partial \Omega \in C^{1,\delta}$ and the inner normal of $\Omega$ at $0$ is parallelled to $e_d$,  the lemma
is deduced from Theorem 2.4 of \cite{DK11} by using a standard localization argument.
\end{proof}

By using the Sobolev embedding theorem and a bootstrap argument, we get
\begin{corollary}
                                    \label{cor3.1}
Let $1\le p<q<\infty$. Assume that $A^{\alpha\beta}$  are partially VMO in $z'$, $u$ is a local weak solution to \eqref{eq3.12} in $\cD_1$ and $u=0$ on $\Gamma_1$, where $f,g\in L_q(\cD_1)$. Then there exist constants $r_0=r_0(d,n,\nu,\Omega,q)\in (0,1)$ and $N=N(d,n,\nu,\Omega,K,\omega_{A,z'},p,q)$ such that
\begin{equation*}
              %                      \label{eq3.09}
\|u\|_{\cH^1_q(\cD_{r_0})}\le N(\|u\|_{L_p(\cD_1)}+\|g\|_{L_q(\cD_1)}+\|f\|_{L_q(\cD_1)}).
\end{equation*}
In particular, if $q>d+2$, it holds that
$$
|u|_{\gamma/2,\gamma;\cD_{r_0}}\le N(\|u\|_{L_p(\cD_1)}+\|g\|_{L_q(\cD_1)}+\|f\|_{L_q(\cD_1)}),
$$
where $\gamma=1-(d+2)/q$.
\end{corollary}

We state a mean oscillation estimate in the special case when the boundary is locally flat and $A^{\alpha\beta}$ depend only on $x^d$.

\begin{lemma}
                                \label{lem3.2}
Let $\gamma\in (0,1)$ and $0< r<R<0$. Assume that $A^{\alpha\beta}=A^{\alpha\beta}(x^d)$ and $w\in \cH^1_2(Q_R^+)$ satisfies
$$
-w_t+D_\alpha(A^{\alpha\beta}D_\beta w)=0
$$
in $Q_R^+$ and $w=0$ on $Q_R\cap \{x^d=0\}$. Denote $W:=A^{d\beta}(x^d)D_\beta w$. Then we have
\begin{align*}
&\int_{Q_r^+}|D_{x'}w|^2+|W-(W)_{Q_r^+}|^2
\,dz\nonumber\\
&\,\le N(r/R)^{d+2+2\gamma}
\int_{Q_R^+}|D_{x'}w|^2+|W-(W)_{Q_R^+}|^2\,dz,
           %                         \label{eq5.35}
\end{align*}
where $N=N(d,n,\nu,\gamma)$.
\end{lemma}
\begin{proof}
The lemma is a simple consequence of (20) in \cite{Dong12} by using odd and even extensions to $Q_R$.
\end{proof}

For the proof of Theorem \ref{thm2}, we consider parabolic systems with coefficients measurable in $(x^{d-1},x^d)$.

\begin{lemma}
                                \label{lem3.5}
Assume that $A^{\alpha\beta}=A^{\alpha\beta}(x^{d-1},x^d)$ and  $u\in \cH^1_2(Q_1)$ satisfies
\begin{equation}
                                   \label{eq:geq}
-u_t+D_\alpha(A^{\alpha\beta}D_\beta u)=0
\end{equation}
in $Q_1$. Then for some $\delta_0=\delta_0(d,n,\nu)\in (0,1)$, we have $u\in C^{\delta_0/2,\delta_0}(Q_{1/2})$ and
\begin{equation}
|u|_{\delta_0/2,\delta_0;Q_{1/2}}\le N\|u\|_{L_2(Q_1)},
                                    \label{eq10.13}
\end{equation}
where $N=N(d,n,\nu)$.
\end{lemma}
\begin{proof}
We may assume that $A^{\alpha\beta}$ are smooth and thus $u$ is smooth. The estimates derived below are independent of the smoothness of $A^{\alpha\beta}$ and u. So we can use the standard mollification argument and then pass to the limit.

By the energy estimate of parabolic systems with measurable coefficients which are independent of $t$ (see, for instance,
\cite[Lemma 3.3]{DK11b}), we have for any $1/2\le r<R\le 1$,
\begin{equation}
                               \label{eq:energyest}
\| u_t\|_{L_2(Q_{r})} +\|D u\|_{L_2(Q_{r})} \leq N \|u\|_{L^2(Q_R)},
\end{equation}
where $N=N(d,n,\nu,r,R)$.
Differentiating the equation with respect to $x''$ and $t$, it is easy to see that, for any nonnegative integers $k$ and $l$, $D_t^k D_{x''}^l u$ are still solutions of \eqref{eq:geq}. Applying the energy estimate \eqref{eq:energyest} to $D_t^k D_{x''}^l u$ gives
\[
\| D_t^{k+1} D_{x''}^l u\|_{L_2(Q_{r})} +\|D D_t^k D_{x''}^l u\|_{L_2(Q_{r})} \leq N \|D_t^k D_{x''}^l u\|_{L^2(Q_R)}.
\]
By iterating the above inequality for $k$ and $l$, we obtain
$$
\|D_t^k D_{x''}^l u\|_{L_2(Q_{r})}+\|D_t^k D_{x''}^l D u\|_{L_2(Q_{r})}\le N\|u\|_{L_2(Q_R)},
$$
where $N=N(d,n,\nu,k,l,r,R)$. It follows from the reverse H\"older inequality (see \cite[Theorem 2.1]{GS}), that
$$
\|D_t^k D_{x''}^l u\|_{L_p(Q_{r})}+\|D_t^k D_{x''}^l D u\|_{L_p(Q_{r})}\le N(d,n,\nu,k,l) \|u\|_{L_2(Q_R)},
$$
where $p=p(d,n,\nu)>2$. This then implies \eqref{eq10.13} with $\delta_0=1-2/p$ by using an anisotropic Sobolev inequality (see, for instance, \cite[Lemma 2.2]{LiNi} or \cite[Lemma 2.1]{Dong13}) with $k+l\ge (d+1)/2$.
\end{proof}

\mysection{Proof of Theorem \ref{thm1}}
                            \label{sec4}

The following lemma reduces the estimate of $[u]_{1/2,1}$ to the estimate of $|Du|_0$.

\begin{lemma}
                                \label{lem4.9}
Let $\varepsilon\in (0,1)$ and $u$ be a weak solution to \eqref{parabolic} in $\cD_1$ with the Dirichlet boundary
condition on $\Gamma_1$. Suppose that $|u|_{0;\cV_{\varepsilon}}<\infty$ and $|Du|_{0;\cV_{\varepsilon}}<\infty$. Then we have
\begin{equation}
                                            \label{eq1.48}
[u]_{1/2,1;\cV_{\varepsilon/2}}\le N(|u|_{0;\cV_{\varepsilon}}+|Du|_{0;\cV_{\varepsilon}}
+|f|_{0;\cV_{\varepsilon}}+|g|_{0;\cV_{\varepsilon}}),
\end{equation}
where $N=N(d,n,\nu, K,  \Omega)$.
\end{lemma}
\begin{proof}
Rewrite \eqref{parabolic} as
$$
-u_t+D_\alpha(A^{\alpha\beta}D_\beta u)=D_\alpha(g_\alpha-B^{\alpha}u)+f-\hat B^{\alpha}D_\alpha u- C u:=D_\alpha \fg_\alpha+\ff,
$$
where
\begin{equation}
                    \label{eq328.20.58}
\fg_\alpha=g_\alpha-B^{\alpha}u,\quad
\ff=f-\hat B^{\alpha}D_\alpha u- C u.
\end{equation}
We fix $z_0\in \overline{\cV_{\varepsilon/2}}$ and take $r\in (0,\varepsilon/2)$. By Lemma \ref{paraPoin} with $\Omega$ replaced by $\cC_\varepsilon$, we have
$$
\int_{Q_r(z_0)\cap \cV_\varepsilon}|u-(u)_{Q_r(z_0)\cap \cV_\varepsilon}|^2\,dz
\le Nr^2\int_{Q_r(z_0)\cap \cV_\varepsilon}\left(|Du|^2+|\fg|^2+r^2|\ff|^2\right)
$$
$$
\le Nr^{n+4}\left(|Du|_{0;\cV_{\varepsilon}}+|u|_{0;\cV_{\varepsilon}}
+|g|_{0;\cV_{\varepsilon}}+|f|_{0;\cV_{\varepsilon}}\right)^2.
$$
The inequality \eqref{eq1.48} then follows from Campanato's characterization of H\"older continuous functions.
\end{proof}

For the proof of Theorem \ref{thm1}, we use an idea in \cite{DK11} to locally ``cut off'' the boundary.

Since $\Omega$ is a $C^{1,\delta}$ domain, there exist constants $\mu$ and $R_0\in (0,1/2)$ depending on $\Omega$ such that for any $R\in (0,R_0)$ we have
$$
 B_R\cap \partial\Omega\subset \{|x^d|\le \mu R^{1+\delta}\},\quad \cC_R\subset \Omega_R,
$$
and, for any $x\in \cC_R$, $B_{|x|/2}(x)\subset \Omega_1$.
We fix an $R\in (0,R_0)$ and take a smooth function $\chi$ defined on $\bR$ such that
$$
\chi(x^d)\equiv 0\quad\text{for}\,\,x^d\le \mu R^{1+\delta},
\quad \chi(x^d)\equiv 1\quad\text{for}\,\,x^d\ge  2\mu R^{1+\delta},
$$
$$
|\chi'|\le \frac 2 \mu R^{-1-\delta}.
$$

\begin{lemma}
Let $\hat u:=\chi u$. Then $\hat u$ vanishes on $\cD_R \cap \{x^d \le \mu R^{1+\delta}\}$ and satisfies in $\hat \cD_{R}:=\cD_{R}\cap \{x^d> \mu R^{1+\delta}\}$,
\begin{equation}
                                    \label{eq17.23b}
\cP_0 \hat u=\Div m+ \chi (\Div \fg+\ff)+h_1+h_2,
\end{equation}
where $\bar A^{\alpha\beta}=A^{\alpha\beta}(0,x^d)$,
$$
\cP_0=-D_t+D_\alpha(\bar A^{\alpha\beta}D_\beta),\quad
m_\alpha^i(z)=(\bar A^{\alpha\beta}_{ij}-A^{\alpha\beta}_{ij})
D_\beta u^j,
$$
and
\begin{equation*}
h_1=D_\alpha\big(\bar A^{\alpha\beta} D_{\beta}((\chi-1)u)\big),\quad
h_2=(1-\chi)D_\alpha(A^{\alpha\beta} D_{\beta} u).
\end{equation*}
\end{lemma}
\begin{proof}
This can be easily seen if one begins with multiplying the equation of $u$ by $\chi$ and then adding $h_2$ to the both sides.
\end{proof}

Now we are ready to prove Theorem \ref{thm1}. Choose a $q>d+2$ sufficiently large such that
$$
1-(d+2)/q>\delta,\quad (d+2+\delta)(1-2/q)>d+2+\delta/2.
$$
Let $r_0$ be the number in Corollary \ref{cor3.1} with this $q$. We take $\delta/4<\gamma<1$, and $0<r<R\le R_1:=\min\{r_0,R_0\}$. Recall the definitions of $\fg$ and $\ff$ in \eqref{eq328.20.58}.
Let $$
u_0(x^d)=\int_{\mu R^{1+\delta}}^{x^d}(\bar A^{dd}(s))^{-1}\fg_d(0,s)\,ds,
\quad u_e=\hat u-u_0.
$$
Clearly, by \eqref{eq17.23b} $u_e$ satisfies
\begin{equation*}
                %                \label{eq21.58}
\cP_0 u_e=
(\chi -1)\Div \fg+\Div(\fg-\fg(0,x^d)+m)+\chi \ff+h_1+h_2,
\end{equation*}
and $u_e=0$ on $\cD_R\cap \{x^d=\mu R^{1+\delta}\}$.

Let $v$ be a weak solution to the equation
\[
\left\{
  \begin{aligned}
    \cP_{0} v= (\chi -1)\Div \fg+\Div(\fg-\fg(0,x^d)+m)+\chi \ff+h_1+h_2 \quad & \hbox{in $\hat \cD_{R}$;} \\
    v=0 \quad & \hbox{on $\partial_p \hat \cD_{R}$.}
  \end{aligned}
\right.
\]
Multiplying the both sides of equation of $v$ by $v$ and integrating by parts, we obtain
\begin{align}
\int_{\hat \cD_R} |D v|^2\,dz \le N &\int_{\hat \cD_R}|\fg D((\chi -1) v)|
+|(\fg-\fg(0,x^d)+m)D v|\nonumber\\&
+|\chi \ff v|+|D((\chi-1)u)Dv|+|D((\chi-1)v)Du|\, dz,
                                \label{eq:energ}
\end{align}
where $N=N(d,n,\nu,K)$.
Since $v=0$ on $\cD_R\cap \{x^d=\mu R^{1+\delta}\}$ and $\chi=1$ for $x^d\ge 2\mu R^{1+\delta}$,  by using the Hardy inequality
\begin{align*}
\int_{\hat \cD_R} |D((\chi -1) v)|^2\, dz &\leq N  \int_{\hat \cD_R\cap \{x^d<2\mu R^{1+\delta}\}} R^{-2(1+\delta)} | v|^2 + |Dv|^2\, dz\\&
\leq N \int_{\hat \cD_R\cap \{x^d<2\mu R^{1+\delta}\}} |x^d-\mu R^{1+\delta}|^{-2} | v|^2 + |Dv|^2\, dz\\
&\leq  N\int_{\hat \cD_R\cap \{x^d<2\mu R^{1+\delta}\}} |Dv|^2\, dz,
\end{align*}
where $N=N(d,n,\mu)$.  Similarly, since $u=0$ on $\Gamma_R$ and $x^d\ge -\mu R^{1+\delta}$ for $x\in \cD_R$, we have
\begin{align*}
\int_{\hat \cD_R} |D((\chi -1) u)|^2\, dz &\leq   \int_{ \cD_R\cap \{x^d<2\mu R^{1+\delta}\}}  |D((\chi -1) u)|^2\, dz \\
&\leq  N\int_{\cD_R\cap \{x^d<2\mu R^{1+\delta}\}} |D u|^2\, dz.
\end{align*}
Together with the Cauchy inequality and Poincar\'e inequality, we then deduce from \eqref{eq:energ} that
\begin{align*}
&\|Dv\|_{L_2(\hat \cD_{R})}+R^{-1} \|v\|_{L_2(\hat \cD_{R})}\\
&\le N \big(\|\fg\|_{L_2(\hat\cD_{R}\cap \{x^d<2\mu  R^{1+\delta}\})}+
\|\fg(z)-\fg(0,x^d)+m\|_{L_2(\hat\cD_{R})}\\
&\quad +R\|\ff\|_{L_2(\hat\cD_{R}))}+  \|Du\|_{L_2(\cD_{R}\cap \{x^d<2\mu  R^{1+\delta}\})}\big).
\end{align*}
It follows from the H\"older continuity of  $A^{d\beta}$, $B^d$, and $g_d$ with respect to $z'$ and the H\"older inequality that
\begin{align}
                                 \label{eq4.34}
&\|Dv\|_{L_2(\hat \cD_{R})}+R^{-1} \|v\|_{L_2(\hat \cD_{R})} \nonumber\\&
\le  NR^{d/2+1+\delta/4}\big(|\fg|_{z',\delta/2,\delta;\cD_{2/3}}+
|f|_{0;\cD_{2/3}}+\|Du\|_{L_q(\cD_{R})}+|u|_{0;\cD_{R}}\big) ,
\end{align}
where we used $(d+2+\delta)(1-2/q)>d+2+\delta/2$.

Let $w:=u_e-v$. Then $w$ satisfies $\cP_{0} w = 0$ in $\hat \cD_{R}$ and $w=0$ on $\cD_R\cap \{x^d=\mu R^{1+\delta}\}$. Denote
\begin{equation*}
                %                \label{eq21.11}
W=\bar A^{d\beta}(x^d)D_{\beta}w,\quad \hat\cD_r=\cD_r\cap \{x^d>\mu R^{1+\delta}\}.
\end{equation*}
By Lemma \ref{lem3.2}, we have
\begin{align}
&\int_{\hat \cD_r}|D_{x'}w|^2+|W-(W)_{\hat \cD_r}|^2\,dz\nonumber\\
&\,\le N(r/R)^{d+2+2\gamma}
\int_{\hat \cD_R}|D_{x'}w|^2+|W-(W)_{\hat \cD_R}|^2\,dz.
                                    \label{eq23.33}
\end{align}

Denote
$U_e:=\bar A^{d\beta}(x^d)D_{\beta}u_e=\bar A^{d\beta}D_\beta \hat u-\fg_d(0,x^d)$.
By the definition of $u_0$ and $u_e$, $D_{x'}u_e=D_{x'}\hat u$. We then combine \eqref{eq4.34} with \eqref{eq23.33} and use the triangle inequality to obtain
\begin{align}
&\int_{\hat\cD_r}|D_{x'}\hat u|^2+|U_e-(U_e)_{\hat \cD_r}|^2\,dz\nonumber\\
&\,= \int_{\hat\cD_r}|D_{x'}u_e|^2+|U_e-(U_e)_{\hat \cD_r}|^2\,dz\nonumber\\
&\,\le N_1(r/R)^{d+2+2\gamma}
\int_{\hat\cD_R}|D_{x'}u_e|^2+|U_e-(U_e)_{\hat\cD_R}|^2\,dz\nonumber\\
&\quad +NR^{d+2+\delta/2}\big(|\fg|_{z',\delta/2,\delta;\cD_{2/3}}+
|f|_{0;\cD_{2/3}}+\|Du\|_{L_q(\cD_{R})}+|u|_{0;\cD_{R}}\big)^2.
                                    \label{eq10.00}
\end{align}
where $N_1=N_1(d,n,\nu,\gamma)$.

Using the H\"older continuity of $A^{d\beta}$, $B^d$, and $g_d$ with respect to $z'$, \eqref{eq10.00} holds with
$$
U:=A^{d\beta}D_\beta u+B^d u-g_d.
$$ in place of $U_e$.
In addition, since $u-\hat u=(1-\chi)u$ is supported in a thin strip $\cD_R\cap \{|x^d|\le 2\mu R^{1+\delta}\}$, by the H\"older inequity
\begin{align*}
\int_{\hat\cD_r}|D_{x'} u|^2 \, dz &\le N \int_{\hat\cD_r}|D_{x'}\hat u|^2\, dz + N \int_{\hat\cD_r\cap \{x^d\le 2\mu R^{1+\delta}\}}|D_{x'} u|^2\\&
\le N \int_{\hat\cD_r}|D_{x'}\hat u|^2\, dz + N R^{d+2+\delta/2}\|D_{x'}u\|_{L_q(\cD_{R})}^2.
\end{align*}
Also,
\begin{align*}
\int_{\hat\cD_R}|D_{x'} u_e|^2 \, dz&=\int_{\hat\cD_R}|D_{x'} \hat u|^2 \, dz\\
&\le N \int_{\hat\cD_R}|D_{x'} u|^2\, dz + N R^{d+2+\delta/2}\|D_{x'}u\|_{L_q(\cD_{R})}^2.
\end{align*}
With these three observations,  we deduce from \eqref{eq10.00} that for any $0<r<R\le R_1$,
\begin{align*}
&\int_{\cD_r}|D_{x'} u|^2+| U-( U)_{\cD_r}|^2\,dz\nonumber\\
&\,\le N_1(r/R)^{d+2+2\gamma}
\int_{\cD_R}|D_{x'} u|^2+|U-(U)_{\cD_R}|^2\,dz\nonumber\\
&\quad +NR^{d+2+\delta/2}\big(|\fg|_{z',\delta/2,\delta;\cD_{2/3}}+
|f|_{0;\cD_{2/3}}+\|Du\|_{L_q(\cD_{R})}+|u|_{0;\cD_{R}}\big)^2.
           %                         \label{eq4.45}
\end{align*}
Since $\delta/2<2\gamma$, by a well-known iteration argument (see, for instance, \cite[Chapter 3.2]{Giaq93}) we obtain for any $r<R_1$,
\begin{align}
&\int_{\cD_r}|D_{x'} u|^2+| U-( U)_{\cD_r}|^2\,dz\nonumber\\
&\,\le Nr^{d+2+\delta/2}\big(|\fg|_{z',\delta/2,\delta;\cD_{2/3}}+
|f|_{0;\cD_{2/3}}+\|Du\|_{L_q(\cD_{R_1})}+|u|_{0;\cD_{R_1}}\big)^2\nonumber\\
&\,\le Nr^{d+2+\delta/2}\big(|g|_{z',\delta/2,\delta;\cD_{2/3}}+
|f|_{0;\cD_{2/3}}+\|u\|_{L_2(\cD_{1})}\big)^2,
                                    \label{eq4.45b}
\end{align}
where in the last inequality we used Corollary \ref{cor3.1}. Shifting the $t$-coordinate in \eqref{eq4.45b}, we have
\begin{align}
&\int_{\cD_r(t_0,0)}|D_{x'} u|^2+| U-( U)_{\cD_r(t_0,0)}|^2\,dz\nonumber\\
&\quad \le Nr^{d+2+\delta/2}\big(|g|_{z',\delta/2,\delta;\cD_1}+
|f|_{0;\cD_{1}}+\|u\|_{L_2(\cD_{1})}\big)^2
                                    \label{eq4.45d}
\end{align}
for any $0<r<R_1$ and $t_0\in [-1/4,0]$.

Now we fix a $z_0=(t_0,x_0)\in \overline{\cV_{R_1/2}}$ and denote $r_{x_0}=|x_0|/2$. By the definition of $R_1$, we have $B_{r_{x_0}}(x_0)\subset \Omega_1$ and $Q_{r_{x_0}}(x_0)\subset \cD_1$. Similar to \eqref{eq4.45d} (see (32) and (33) of \cite{Dong12}), for any $r\le r_{x_0}$, it holds that
\begin{align}
&\int_{Q_r(z_0)}|D_{x'} u-(D_{x'} u)_{Q_r(z_0)}|^2+| U-( U)_{Q_r(z_0)}|^2\,dz\nonumber\\
&\,\le N(r/r_{x_0})^{d+2+2\delta}\int_{Q_{r_{x_0}}(z_0)}|D_{x'} u-(D_{x'} u)_{Q_{r_{x_0}}(z_0)}|^2+| U-( U)_{Q_{r_{x_0}}(z_0)}|^2\,dz\nonumber\\
&\,\,\,+Nr^{d+2+2\delta}\big(|g|_{z',\delta/2,\delta;\cD_1}+
|f|_{0;\cD_{1}}+\|u\|_{L_2(\cD_{1})}\big)^2\nonumber\\
&\,\le N(r/r_{x_0})^{d+2+\delta/2}\int_{Q_{r_{x_0}}(z_0)}|D_{x'} u|^2+| U-( U)_{Q_{r_{x_0}}(z_0)}|^2\,dz\nonumber\\
&\,\,\,+Nr^{d+2+\delta/2}\big(|g|_{z',\delta/2,\delta;\cD_1}+
|f|_{0;\cD_{1}}+\|u\|_{L_2(\cD_{1})}\big)^2.
                                    \label{eq4.45c}
\end{align}
Because $Q_{r_{x_0}}(z_0)\subset \cD_{3r_{x_0}}(t_0,0)$ and $3r_{x_0}<R_1$, using \eqref{eq4.45d} the first integral on the right-hand side of \eqref{eq4.45c} is bounded by
\begin{align*}
&N(r/r_{x_0})^{d+2+\delta/2}\int_{\cD_{3r_{x_0}}(t_0,0))}|D_{x'} u|^2+| U-( U)_{\cD_{3r_{x_0}}(t_0,0)}|^2\,dz\\
&\,\le Nr^{d+2+\delta/2}\big(|g|_{z',\delta/2,\delta;\cD_1}+
|f|_{0;\cD_{1}}+\|u\|_{L_2(\cD_{1})}\big)^2.
\end{align*}
On the other hand, for any $r\in (r_{x_0},R_1/3)$, we have $Q_{r}(z_0)\subset \cD_{3r}(t_0,0)$ and $3r<R_1$. Therefore, by \eqref{eq4.45d},
\begin{align*}
&\int_{Q_{r_{x_0}}(z_0)}|D_{x'} u-(D_{x'} u)_{Q_{r_{x_0}}(z_0)}|^2+| U-( U)_{Q_{r_{x_0}}(z_0)}|^2\,dz\\
&\,\le N \int_{\cD_{3r}(t_0,0)}|D_{x'} u|^2+| U-(U)_{\cD_{3r}(t_0,0)}|^2\,dz\\
&\,\le Nr^{d+2+\delta/2}\big(|g|_{z',\delta/2,\delta;\cD_1}+
|f|_{0;\cD_{1}}+\|u\|_{L_2(\cD_{1})}\big)^2
\end{align*}
By Campanato's characterization of H\"older continuous functions, $D_{x'}u$ and $U$ are H\"older continuous in $\overline{\cV_{R_1/2}}$.
Theorem \ref{thm1} then follows immediately from Lemma \ref{lem4.9}.

\mysection{Proof of Theorem \ref{thm2}}
                                           \label{sec5}

We shall use the following interior H\"older estimate, which is of independent interest. In the 2D elliptic case, this result is classical, which follows directly from the reverse H\"older inequality and the Sobolev embedding theorem.

\begin{proposition}
                    \label{prop5.1}
Suppose that $A$ is partially VMO with respect to $z''$, $g\in L_{p_0}(Q_1)$ for some $p_0>d+2$, and $f\in L_{p_0^*}(Q_1)$, where
$$p_0^*=p_0(d+2)/(p_0+d+2).$$
Let $u$ be a weak solution to \eqref{parabolic} in $Q_1$. Let $\delta_0$ be the constant in Lemma \ref{lem3.5} and $\delta_1=\min(\delta_0,1-(d+2)/p_0)$. Then for any $\delta\in (0,\delta_1)$, we have $u\in C^{\delta/2,\delta}(\overline{Q_{1/2}})$. Moreover,
\begin{equation*}
                                    %\label{eq9.24c}
|u|_{\delta/2,\delta;Q_{1/2}}
\le N(\|g\|_{L_{p_0}(Q_1)}+\|f\|_{L_{ p_0^*}(Q_1)}
+\|u\|_{L_2(Q_1)}),
\end{equation*}
where $N=N(d,n,\delta,\nu,K,p_0,\omega_{A,z''})$.
\end{proposition}
\begin{proof}
Let $R_0\in (0,1/8)$ be a number to be specified later. We take $0<r<R\le R_0$. First, note that as a weak solution of a linear uniformly parabolic system,
$u$ satisfies the following improved integrability estimate
for some $p=p(d,n,\nu)>2$ (see \cite[Theorem 2.1]{GS}):
\begin{equation}
            \label{eq5.3}
\Big(\dashint_{Q_R}|D u|^p\,dz\Big)^{2/p} \le
N\,\dashint_{Q_{2R}}|D u|^2\,dz.
\end{equation}

By the remark at the end of Section \ref{ssec2.3}, we can find $\bar A=\bar A(x^{d-1},x^d)$ depending on $R$ such that
\begin{equation}
                    \label{eq2.03}
\dashint_{Q_R}|A(t,x'',x^{d-1},x^d)-
\bar A(x^{d-1},x^d)|\,dz\le N(d)\,\omega_{A,z''}(\sqrt 2 R).
\end{equation}
We decompose $u=v+w$, where $v$ is the weak solution of
\begin{align}
                                    \label{eq:v}
-v_t+D_\alpha\big(\bar{A}^{\alpha\beta} D_\beta v\big)=0
\quad\text{in }Q_R
\end{align}
with $v=u$ on $\partial_p Q_R$, and $w$ is the weak solution of
\begin{align}
\nonumber
&-w_t+D_\alpha\big(\bar{A}^{\alpha\beta}D_\beta w\big)\\
                \label{eqn:A-10}
&\,\,=D_\alpha\big((\bar {A}^{\alpha\beta}-A^{\alpha\beta})
D_\beta u\big)
-D_\alpha(B^{\alpha}u)-\hat B^{\alpha}D_\alpha u- C u+\Div g + f
\end{align}
in $Q_R$
with $w=0$ on $\partial_p Q_R$.

Since $B^\alpha$, $\hat B^\alpha$, and $C$ are bounded by $K$, multiplying both sides of \eqref{eqn:A-10} by $w$ and integrating by parts we get
\begin{align*}
\int_{Q_R}R^{-2}|w|^2+|Dw|^2\,dz & \leq
N\int_{Q_R}|A-\bar A| |Du Dw|+|u Dw|+|g Dw|\, dz \\ & \quad +N\int_{Q_R}(|Du|+|u|+|f|)|w|\,dz,
\end{align*}
where $N=N(d,n,\nu,K)$.
It follows from the Cauchy inequality and the Poincar\'e inequality that
\begin{equation}
\begin{split}
\int_{Q_R}|Dw|^2\,dz
\le & N\int_{Q_R}|A-\bar A|^2 |Du|^2+|u|^2+|g|^2\, dz\\& +N\int_{Q_R} R^2(|Du|^2+|f|^2)\,dz.
               \label{eq:1}
\end{split}
\end{equation}
By Lemma \ref{paraPoin}, the H\"older inequality, \eqref{eq5.3}, and \eqref{eq2.03}, from \eqref{eq:1} we have
\begin{align*}
&\int_{Q_R}|w-(w)_{Q_R}|^2\,dz \\
&\leq
NR^2 \int_{Q_R}|A-\bar A|^2 |Du|^2+|u|^2+|g|^2+R^2(|Du|^2+f^2)\,dz\\
&\leq
NR^2 \big(\int_{Q_R} |A-\bar A|^q \,dz\big)^{2/q}\big(\int_{Q_R} |Du|^p\,dz\big)^{2/p}\\
&\quad+NR^2 \int_{Q_R}|u|^2+|g|^2+R^2(|Du|^2+f^2)\,dz\\
&\leq N R^2(\omega_{A,z''}^{2/q}(\sqrt 2 R)+R^2)
\int_{Q_{2R}}|D u|^2\,dz\\
&\quad+R^{2+(d+2)(1-2/p_0)}\big(\|u\|^2_{L_{p_0}(Q_{1/8})}+\|g\|^2_{L_{p_0}(Q_{1/8})}
+\|f\|^2_{L_{p^*_0}(Q_{1/8})}\big),
\end{align*}
where
$$
q=2p/(p-2),\quad 2+(d+2)(1-2/p_0)>d+2.
$$
It follows from the Caccioppoli inequality that
\begin{align}
&\int_{Q_R}|w-(w)_{Q_R}|^2\,dz \nonumber\\
&\leq N (\omega_{A,z''}^{2/q}(\sqrt 2 R)+R^2)
\int_{Q_{4R}}|u-(u)_{Q_{4R}}|^2\,dz\nonumber\\
                    \label{eq5.4}
&\quad+R^{2+(d+2)(1-2/p_0)}\big(\|u\|^2_{L_{p_0}(Q_{1/2})}+\|g\|^2_{L_{p_0}(Q_{1/2})}+\|f\|^2_{L_{p^*_0}(Q_{1/2})}\big).
\end{align}
Since $v-(v)_{Q_R}$ also satisfies \eqref{eq:v}, applying Lemma \ref{lem3.5} with a proper scaling to $v-(v)_{Q_R}$ yields
\begin{equation}
                    \label{eqn:A-11}
\int_{Q_r}|v-(v)_{Q_r}|^2\,dz\leq
N(r/R)^{d+2+2\delta_0} \int_{Q_R}|v-(v)_{Q_R}|^2\,dz,
\end{equation}
for some constant $N=N(d,n,\nu)$. Indeed, we may assume $r<R/2$, otherwise it holds trivially. By a direct computation, we have
\begin{align*}
\int_{Q_r}|v-(v)_{Q_r}|^2\,dz & \leq \int_{Q_r}|v-v(0,0)|^2\,dz\leq r^{d+2+2\delta_0} [v]^2_{\delta_0/2,\delta_0;Q_r}\\
&= r^{d+2+2\delta_0} [v-(v)_{Q_R}]^2_{\delta_0/2,\delta_0;Q_r}\\
& \leq N (r/R)^{d+2+2\delta_0} \int_{Q_R}|v-(v)_{Q_R}|^2\,dz.
\end{align*}

Combining \eqref{eq5.4} and \eqref{eqn:A-11} and using the triangle inequality, we see that
\begin{align*}
&\int_{Q_r}|u-(u)_{Q_r}|^2 \,dz\\
&\leq
N\big((r/R)^{d+2+2\delta_0}+\omega_{A,z''}^{2/q}(\sqrt 2 R)+R^2\big)
\int_{Q_{4R}}|u-(u)_{Q_{4R}}|^2\,dz\\
&\quad +R^{2+(d+2)(1-2/p_0)}\big(\|u\|^2_{L_{p_0}(Q_{1/2})}+\|g\|^2_{L_{p_0}(Q_{1/2})}+\|f\|^2_{L_{p^*_0}(Q_{1/2})}\big).
\end{align*}
If we choose $R_0$ sufficiently small so that $\omega_{A,z''}^{2/q}(\sqrt 2 R_0)+R_0^2$
is small, then by a well-known iteration argument (see, for instance, \cite[Chapter 3.2]{Giaq93}) we obtain for any $\delta\in (0,\delta_1)$,
\begin{equation*}
\int_{Q_r}|u-(u)_{Q_r}|^2\,dz \\
\leq Nr^{d+2+2\delta}\big(\|u\|^{2}_{L_{p_0}(Q_{1/2})}+\|g\|^2_{L_{p_0}(Q_{1/2})}+\|f\|^2_{L_{p^*_0}(Q_{1/2})}\big).
\end{equation*}
Shifting the coordinates, we get for any $z_0\in \overline{Q_{1/2}}$ and $r\in (0,1/8)$,
\begin{equation*}
\int_{Q_r(z_0)}|u-(u)_{Q_r(z_0)}|^2\,dz \\
\leq Nr^{d+2+2\delta}\big(\|u\|^2_{L_{p_0}(Q_{1})}+\|g\|^2_{L_{p_0}(Q_{1})}
+\|f\|^2_{L_{p^*_0}(Q_{1})}\big).
\end{equation*}
By Campanato's characterization of H\"older continuous functions, we conclude that $u\in C^{\delta/2,\delta}(\overline{Q_{1/2}})$ and
\[
[u]_{\delta/2,\delta;Q_{1/2}}\leq \big(\|u\|_{L_{p_0}(Q_{1})}+\|g\|_{L_{p_0}(Q_{1})}
+\|f\|_{L_{p^*_0}(Q_{1})}\big).
\]
The proposition then follows by using a standard iteration argument. See, for instance, \cite[pp. 81]{Giaq93}.
\end{proof}

Next we consider parabolic systems with the same type of the coefficients in a half cylinder. Denote
$$
Q_r^+=Q_r\cap \{x^d>0\},\quad \Gamma_r^+=Q_r\cap \{x^d=0\}.
$$
By using odd and even extensions, we deduce the following corollary from Proposition \ref{prop5.1}.

\begin{corollary}
                            \label{cor5.2}
Suppose that $A$ is partially VMO with respect to $z''$, $g\in L_{p_0}(Q_1^+)$ for some $p_0>d+2$, and $f\in L_{p_0^*}(Q_1^+)$, where $p_0^*=p_0(d+2)/(p_0+d+2)$. Let $u$ be a weak solution to \eqref{parabolic} in $Q_1^+$ with the zero Dirichlet boundary condition on $\Gamma_1$. Let $\delta_0$ be the constant in Lemma \ref{lem3.5} and $\delta_1=\min(\delta_0,1-(d+2)/p_0)$. Then for any $\delta\in (0,\delta_1)$, we have $u\in C^{\delta/2,\delta}(\overline{Q_{1/2}^+})$. Moreover,
\begin{equation*}
              %                      \label{eq9.24c1}
|u|_{\delta/2,\delta;Q_{1/2}^+}
\le N(\|g\|_{L_{p_0}(Q_1^+)}+\|f\|_{L_{ p_0^*}(Q_1^+)}
+\|u\|_{L_2(Q_1^+)}),
\end{equation*}
where $N=N(d,n,\delta,\nu,K,p_0,\omega_{A,z''})$.
\end{corollary}

We are now ready to complete the proof of Theorem \ref{thm2}.

Let $\cS$ be the plane spanned by $e_{d-1}$ and $e_d$ so that the normal direction at $x_0$ lies on $\cS$.
Let $\bar y^1,\dots, \bar y^d$  be an orthogonal coordinate system centered at $x_0$ such that
$\bar y^{d-1}=x^d-x^d_0$, and $\bar y^{d}$ is a direction on $\cS$, which is orthogonal to $x^d$.
Since $\partial \Omega \in C^{1,\delta}$ and the inner normal of $\partial \Omega $ at $x_0$ is not paralleled to $x^d$,
locally $\Omega$ can be represented as a $C^{1,\delta}$ graph $( \bar y^1,\dots,  \bar y^{d-1}, \phi( \bar y^1,\dots, \bar y^{d-1}))$
near $x_0$.  We make a change of variables $y^d=\bar y^d-\phi (\bar y^1,\dots, \bar y^{d-1})$
and $y^i=\bar y^i$ for $i=1,\dots, d-1$ to flatten the boundary near $x_0$ so that the parabolic system is
defined on $(-1,0)\times B_{2\epsilon}^+$ and the new leading coefficients are measurable in $y^{d-1}$ and VMO with respect to the other directions.
Then the theorem follows immediately from Corollary \ref{cor5.2}.
%\section*{Acknowledgement}

%========================================================================

\end{document}